\documentclass[12pt,psamsfonts]{amsart}
\usepackage{amsmath}
\usepackage{amsthm}
\usepackage{amssymb}
\usepackage{amscd}
\usepackage{amsfonts}
\usepackage{amsbsy}
\usepackage{graphicx}
\usepackage[dvips]{psfrag}
\usepackage{array}
\usepackage{color}
\usepackage{epsfig}
\usepackage{url}
\usepackage{overpic}
\usepackage{epstopdf}
\usepackage{scalefnt}

\newcolumntype{L}{>{\displaystyle}l}
\newcolumntype{C}{>{\displaystyle}c}
\newcolumntype{R}{>{\displaystyle}r}

\newcolumntype{L}{>{\displaystyle}l}
\newcolumntype{C}{>{\displaystyle}c}
\newcolumntype{R}{>{\displaystyle}r}

\newtheorem {theorem} {Theorem} 
\newtheorem {proposition} [theorem] {Proposition}

\newtheorem {lemma} [theorem] {Lemma}

\newtheorem {example} {Example}

\newtheorem {remark}[theorem]{Remark}

\begin{document}

\title[continuous piecewise linear differential systems]
{LIMIT CYCLES BIFURCATING FROM A PERIOD ANNULUS\\ IN  CONTINUOUS PIECEWISE LINEAR DIFFERENTIAL \\ SYSTEMS WITH THREE ZONES}

\let\thefootnote\relax\footnotetext{Corresponding author Claudio Pessoa: Departamento de Matem\'{a}tica, IBILCE--UNESP, 15054--000 S. J. Rio Preto,
S\~ao Paulo, Brazil. email:
pessoa@ibilce.unesp.br}

\author[M. F. S. Lima, C. Pessoa and W. F. Pereira]
{Maur\'{\i}cio Firmino Silva Lima$^{1}$, Claudio Pessoa$^{2}$ and Weber F. Pereira$^{2}$}

\address{$^1$ Centro de Matem\'atica Computa\c c\~ao e Cogni\c c\~ao, Universidade Federal do ABC\\
Santo Andr\'e, S\~ao Paulo, 09210-170, Brazil} \email{mauricio.lima@ufabc.edu.br}

\address{$^2$  Departamento de Matem\'{a}tica, IBILCE--UNESP \\ J. Rio Preto,
S\~ao Paulo, 15054--000, Brazil} \email{pessoa@ibilce.unesp.br}\email{weberf@ibilce.unesp.br}

\subjclass[2010]{34A36, 34C29, 37G15}

\keywords{piecewise linear vector fields, Poincar\'e map, limit cycles, center, focus}

\maketitle

\begin{abstract}
We study a class of planar
continuous piecewise linear vector fields with three zones. Using
the Poincar\'{e} map and some techniques for proving the existence of limit cycles for smooth differential systems, we prove that this class admits at least two limit cycles that appear by perturbations of a period annulus. Moreover, we describe the bifurcation of the limit cycles for this class through two examples of two-parameter families of  piecewise linear vector fields with three zones.  
\end{abstract}

\section{Introduction}

One of the most important and studied problem in the qualitative theory of differential systems, in particular for planar systems, is the maximum number, stability and position of limit cycles, see for instance \cite{Yan} and \cite{Zou}.

The famous second part of the 16th Hilbert problem proposed in the $1900^\prime$s and that have been studied for a long time addresses the problem of the maximum number and position of limit cycles restricted to the planar polynomial differential equations.

In the piecewise continuous context this problem has been studied by many authors, see for instance \cite{Freire}, \cite{Hogan}, \cite{Llibre3} and \cite{van}. This class of system is very
important, mainly by their numerous applications, where they arise
in a natural way, for instance in control theory \cite{Lefschetz}
and \cite{Narendra}, design of electric circuits \cite{Chua},
neurobiology \cite{FitzHugh} and \cite{Nagumo}, etc.

In this context, specially for piecewise linear differential systems, many works have been developed. Most of them obtaining results on the existence and uniqueness of limit cycles for systems with  only one curve of discontinuity. For systems with more then one curve of discontinuity not many works are available and, more important than this, recently (see \cite{Llibre4}) an example with more then one limit cycle could be obtained for a special class of Liénard piecewise linear differential system with two curves of discontinuity.

As far as we know the paper \cite{Llibre4} is, until this moment, the only place where at least two limit cycles was obtained in this context. 

It is important to observe that the characterization of all possible cases with two (or maybe more than two) limit cycles is far from being completely solved.

In this paper we give a contribution in this direction where we provide a family of piecewise linear differential systems with at least two limit cycles. We observe that the bifurcation that give rise to the second limit cycle is very close to the one that appear in \cite{Llibre4}, namely, one limit cycle visiting the three zones and the second limit cycle visiting two zones and that bifurcates of a period annulus.

We observe that in \cite{Llibre4} this period annulus is obtained of a center and the second limit cycle appears when this center becomes a focus stable/unstable under parameter changes while, in the present paper, this region is obtained from two foci of different stability with eigenvalues of the same modulus and the second limit cycle appears when these foci have attracting/repelling with different magnitude. However, in both cases, the piecewise linear differential systems have a center. In \cite{Llibre4} this center is located in the central region while in the present work the center is supposed to be in the left region.

In order to set the problem, consider the plane $\mathbb R^2$ divided in three closed regions $R_-,$ $R_o$ and $R_+$ which frontier are given by two parallel straight lines $L_-$ and $L_+$  symmetric with respect to
the origin such that $(0,0)\in R_o$ and the regions $R_-$ and
$R_+$ have as boundary the respective straight lines $L_-$ and
$L_+$. In this paper, we study the existence of limit cycles for the family of differential systems

\begin{equation}
{\bf x'}=\left\{\begin{array}{ll}
                       A_-{\bf x}+B_- & {\bf x}\in R_-, \\
                       A_o{\bf x}+B_o & {\bf x}\in R_o, \\
                       A_+{\bf x}+B_+ & {\bf x}\in R_+,
                     \end{array}
\right.\label{system 1}
\end{equation}
that are continuous piecewise linear differential systems with
tree zones where ${\bf x}=(x,y)\in\mathbb{R}^2$, $A_i\in\mathcal{M}_2(\mathbb{R}),$
$B_i\in\mathbb{R}^2,\;i\in\{-,o,+\}$, and ${\bf x'}=\dfrac{d{\bf
x}}{dt}$ with $t$ the time.


Let $X_i=A_i{\bf x}+B_i$, $i\in \{-, o, +\}$, the linear vector
fields given in \eqref{system 1}. Denote by $X$ the
continuous piecewise linear vector field associated to system
\eqref{system 1}, i.e. ${\bf x'}=X({\bf x})$, where ${\bf
x}=(x,y)\in\mathbb{R}^2$. Therefore $X\mid_{R_i}=X_i\mid_{R_i}$,
$i\in \{-, o, +\}$, i.e., $X$ restrict to each one of these zones
are linear systems with constant coefficients that are glued
continuously at the common boundary.

Our goal is improve the results of the papers \cite{Mauricio} and \cite{ClauWeMau}
considering the case not covered in these papers and where at least two limit cycles can be present. Thus, as in \cite{Mauricio} and \cite{ClauWeMau}, we
assume that $X_o$ has an equilibrium of focus type and the
equilibria of $X_-$ and $X_+$ are a center or a focus. From the
next result we can refine these assumptions a bit more, but before
we need some notation.

We say that the vector field $X_i$ has a \textit{real equilibrium}
${\bf x}^*$ in $R_i$ with $i\in\{-,o,+\}$ if ${\bf x}^*$ is an equilibrium of
$X_i$ and ${\bf x}^*\in R_i.$ Otherwise, we will say that $X_i$ has a
\textit{virtual equilibrium} ${\bf x}^*$ in $R_i^c$ if ${\bf x}^*$ is an
equilibrium of $X_i$ and ${\bf x}^*\in R_i^c$, where $R_i^c$ denotes the
complementary of $R_i$ in $\mathbb{R}^2$.

From now on for $i\in\{-,o,+\}$ we denote by $d_i$  the determinant of the
matrix $A_i$ and by $t_i$ its trace. Also we denote by ${\mbox Int}(\Gamma)$ the open region limited by a closed Jordan curve $\Gamma.$

The next result is an immediate consequence of the Green's Formula (see also Proposition 3 of \cite{Llibre2}).

\begin{proposition}
If system {\rm (\ref{system 1})} has a simple invariant
closed curve $\Gamma$ then
$$\int\!\!\!\int_{\mbox{
Int}_-(\Gamma)}t_-dxdy+\int\!\!\!\int_{\mbox{
Int}_o(\Gamma)}t_odxdy+ \int\!\!\!\int_{\mbox{
Int}_+(\Gamma)}t_+dxdy =t_-S_-+t_oS_o+t_+S_+=0,$$ 
where $\mbox{
Int}_i(\Gamma)=\mbox{ Int}(\Gamma)\cap R_i$ and $S_i=\mbox{
area}(\mbox{Int}_i(\Gamma))$ with $i\in\{-,o,+\}.$\label{prop
green}
\end{proposition}

From this result if $X_o$ has a focus and either $X_-$ or $X_+$ has a center then a necessary condition for the existence of such $\Gamma$ is that $\Gamma$ visit at least the two zones having a real or virtual focus with different stability. Then, from now on, we assume the following hypothesis:

\begin{table}[h]
\begin{tabular}{ll}
(H1) &  $X_o$ has a focus. \\ \\
(H2) & The others equilibria of $X_-$ and
$X_+$ are a center and a focus with  \\ & different stability with
respect to the focus of $X_o$.
 \end{tabular}
 \end{table}

As usual a {\it limit cycle} $\mathcal{C}$ of (\ref{system 1}) is an isolated periodic
orbit of (\ref{system 1}) in the set of all periodic
orbits of (\ref{system 1}). We say that $\mathcal{C}$ is {\it hyperbolic} if
the integral of the divergence of the system along it is different
from zero, for more details see for instance \cite{dumortier}.

The main result of this paper is the following.

\begin{theorem}
\label{the_paper2:01} Assume that system {\rm (\ref{system 1})} satisfies
assumptions {\rm (H1)}, {\rm (H2)} and $X_o$ has a real focus at
the boundary of $R_o$. If the focus of $X_o$ belongs to $L_+$ (respectively $L_-$) and $X_+$ (respectively $X_-$) also
has a focus at the same point of $L_+$ (respectively $L_-$) and both the foci give rise to a center for system
\eqref{system 1}, then at least two limit cycles can appear by small perturbations of the parameters of system \eqref{system 1}.
\end{theorem}

\begin{remark} 
Explicit conditions for the existence of at least two limit cycle in terms of the parameters of system {\rm (\ref{system 1})} are techniques, depend of some ingredients defined in the paper, and are summarized in Proposition \ref{teo two limit cycle} of Section \ref{sec Lim b_2< -1}.
\end{remark}

The paper is organized as follows. In Section \ref{sec Norm Form}
we obtain a normal form to system \eqref{system 1} that simplifies
the computations. In Section \ref{sec Poin Map} we study the
behavior of the Poincar\'e maps associated to system \eqref{system 1} that will be used to study the problem and at the Section
\ref{sec Lim b_2< -1} we prove Theorem \ref{the_paper2:01}. Finally, in Section \ref{sec examples} we give two examples that resume the bifurcations of limit cycles to system \eqref{system 1} under hypothesis (H1) and (H2). 

\section{Normal Form}
\label{sec Norm Form}

The following result, proved in  \cite{Mauricio}, give us a
convenient normal form  to write system (\ref{system 1}) with the
number of parameters reduced, and so it will be useful to obtain
the Poincar\'{e} return map and its derivatives.

\begin{lemma}  Suppose that system (\ref{system 1}) is such that $d_o>0.$ Then there exists
a linear change of coordinates that writes system {\rm
(\ref{system 1})} into the form $\dot{{\bf x}}=X({\bf x}),$
with $L_-=\{x=-1\},$ $L_+=\{x=1\},$
$R_-=\{(x,y)\in\mathbb{R}^2;\;x\leq-1\},$
$R_o=\{(x,y)\in\mathbb{R}^2;\;-1\leq x\leq1\},$
$R_+=\{(x,y)\in\mathbb{R}^2;\;x\geq1\}$ and

\begin{equation}
X({\bf x})=\left\{\begin{array}{ll}
                      A_-{\bf x}+B_- & {\bf x}\in R_-, \\
                      A_o{\bf x}+B_o & {\bf x}\in R_o, \\
                      A_+{\bf x}+B_+ & {\bf x}\in R_+,
                    \end{array}\right.\label{base system nf}
\end{equation}
where $A_-=\left(\begin{array}{cc}
                       a_{11} & -1 \\
                       1-b_2+d_2 & a_1
                     \end{array}\right),$ $B_-=\left(\begin{array}{c}
                                                       a_{11} \\
                                                       d_2
                                                     \end{array}\right),$
$A_o=\left(\begin{array}{cc}
                       0 & -1 \\
                       1 & a_1
                     \end{array}\right),$ $B_o=\left(\begin{array}{c}
                                                       0 \\
                                                       b_2
                                                     \end{array}\right),$ $A_+=\left(\begin{array}{cc}
                       c_{11} & -1 \\
                       1+b_2-f_2 & a_1
                     \end{array}\right)$ and $B_+=\left(\begin{array}{c}
                                                       -c_{11} \\
                                                       f_2
                                                     \end{array}\right).$ The dot denotes derivative with respect
to a new time $s.$\label{nf}
\end{lemma}

We call {\it contact point} to a point of a straight line where the vector field is tangent to it.

The next lemma, which proof can be found in \cite{ClauWeMau}, characterizes the equilibria of the vector fields
$X_i$, $i\in \{-, o, +\}$, in real or virtual with respect the
sign of the parameter $b_2$ of the normal form \eqref{base system
nf}.

\begin{lemma}\label{equilibrium}
In the coordinates given by Lemma \ref{nf}, there is a unique
contact point of system {\rm (\ref{base system nf})} with each one
of the straight lines $L_-$ and $L_+$. These points are
respectively $p_-=(-1,0)$ and $p_+=(1,0).$ Moreover under the
assumptions {\rm (H1)} and {\rm (H2)}, we have
\begin{itemize}
\item[(i)] if $b_2<-1$, the equilibrium points of $X_-$ and $X_o$ are virtual
and the equilibrium point of $X_+$ is real;

\item[(ii)] if $b_2=-1$, $X_o$ and $X_+$ have an equilibrium point at $(1,0)$, and $X_-$ has a virtual equilibrium point;

\item[(iii)] if $|b_2|<1$, the equilibrium points of $X_-$ and $X_+$ are virtual
and the equilibrium point of $X_o$ is real;

\item[(iv)] if $b_2=1$, $X_-$ and $X_o$ have an equilibrium point at $(-1,0)$, and $X_+$ has a virtual equilibrium point;

\item[(v)] if $b_2>1$, the equilibrium points of $X_o$ and $X_+$ are virtual
and the equilibrium point of $X_-$ is real.
\end{itemize}
\end{lemma}

From Lemmas \ref{nf} and \ref{equilibrium} we can divide the problem of study the existence of limit cycles for system \eqref{system 1} with hypothesis (H1) and (H2) in  the cases given in the table below.
\medskip
\begin{table}[h]
\scalefont{0.9}
\caption{All the cases under the hypothesis (H1) and (H2).}
{\begin{tabular}{|c|c|c|c|c|c|}
\hline
& Case $b_2<-1$ & Case $b_2=-1$ & Case $|b_2|<1$ & $b_2=1$ & $b_2>1$ \\
 \hline
 Equilibrium of $X_-$ & Virtual & Virtual & Virtual & $(-1,0)$ & Real \\
  \hline
 Equilibrium of $X_o$ & Virtual & $(1,0)$ & Real & $(-1,0)$ & Virtual \\
  \hline
 Equilibrium of $X_+$ & Real & $(1,0)$ & Virtual & Virtual & Virtual \\
 \hline
 \end{tabular}}
 \end{table}

The Case $|b_2|<1$ has been studied in \cite{Mauricio}, in this paper it is showed that  under the hypothesis (H1) and (H2) and additionally assuming that $X_o$ has a real
 focus in the interior of $R_o$, system \eqref{system 1} has a unique limit cycle. Some another cases have been studied in  \cite{ClauWeMau}, where it is showed that under the hypothesis (H1) and (H2) and additionally assuming that
  \begin{itemize}
  \item[(a)] $X_o$ has a virtual focus and
$X_+$ (respectively $X_-$) has a real center, system {\rm (\ref{system 1})} has a
unique limit cycle, which is hyperbolic;
\item[(b)]  $X_o$ has a real focus at
the boundary of $R_o$, system {\rm (\ref{system 1})} has a
unique limit cycle, which is hyperbolic. Except when the focus of $X_o$ belongs to $L_+$ (respectively $L_-$) and $X_+$ (respectively $X_-$) also
has a focus and both foci give rise to a center for system
\eqref{system 1}. In this case system \eqref{system 1} has no
limit cycles.
 \end{itemize}
Note that,  the Cases $b_2>1$ and $b_2=1$ are equivalent, by a rotation through of  an angle of $180^o$, to the Cases $b_2<-1$ and $b_2=-1$, respectively.  Therefore, from above results, remains only the case $b_2<-1$ with $X_+$ having a real focus to be studied. Here, Theorem \ref{the_paper2:01} shows that we have at least two limit cycles in this case. A complete study of the case $b_2<-1$, when $X_+$ has a real focus, it is still an open problem.

\section{Poincar\'{e} Return Map}
\label{sec Poin Map}

In this section we will rewrite the problem of finding limit cycles that visit the three zones $R_i$, $i\in\{-, o, +\}$ or even only two of them in terms of finding the fixed points of an appropriated Poincar\'{e} return map. This Poincar\'{e} return map will be the composition of four or two different Poincar\'{e} maps, according to the number of zones. As in \cite{ClauWeMau}, these Poincar\'{e} maps will be defined in the transversal sections $L_\pm^{O}=\{(\pm1,y);\,y\geq0\}$ and $L_\pm^{I}=\{(\pm1,y);\,y\leq0\}.$ In order to study the qualitative behavior of each one of these maps, we will do a convenient parameterization in the transversal sections $L_\pm^{O}$ and $L_\pm^{I}$  assuming $b_2\neq \pm 1$. We parametrize $L_-^{O}$ by the
parameter $c$ defined as follows. Let $p_-=(-1,0)$ be the contact point of $X_-$ with $L_-$ and
$\dot{p}_-=X_-(p_-)=(0,b_2-1).$ Given, $p\in L_-^{O}$  we take $c\geq 0$ as the unique non-negative real satisfying  $p=p_--c \dot{p}_-.$ Similarly we parametrize
$L_-^{I}$ by the parameter $d$, i.e. given $p\in L_-^{I}$ we take $d\geq 0$ as the unique non-negative real satisfying $p=p_-+d \dot{p}_-$.

In a very close way we parametrize $L_+^{O}$ by the parameter $b$ as
follows. Let $p_+=(1,0)$ be the contact point of $X_+$ with $L_+$ and $\dot{p}_+=X_+(p_+)=(0, b_2 + 1).$ Given, $p\in L_+^{O}$ we take $b\geq 0$ as the unique non-negative real stisfying $p=p_+-b \dot{p}_+.$ In a
similar way we parametrize $L_+^{I}$ by the parameter $a,$ i.e.
given $p\in L_+^{I}$  we take $a\geq 0$ as the unique non-negative real satisfying
$p=p_++a \dot{p}_+$.

{For} study the limit cycles of system (\ref{base system nf}) that visit the three zones $R_i$, $i\in\{-, o, +\}$, the Poincar\'{e} return map $\Pi$ is defined on $L_-^O$ and its fixed points correspond to these limit cycles.  This map involves all the vector fields $X_i$, $i\in \{-, o, +\}$, and has the form
\[
\Pi=\bar{\pi}_o\circ\pi_+\circ \pi_o\circ\pi_-,
\]
where the Poincar\'{e} maps $\bar{\pi}_o$, $\pi_+$, $\pi_o$, and $\pi_-$ will be defined as follows.

The vector field $X_-$ points
toward the region $R_-$ in $L_-^O$ while in $L_-^{I}$ the vector field points
toward the region $R_o.$ Then we can define a Poincar\'{e} map $\Pi_-$ from $L_-^{O}$ to $L_-^{I}$ by $\Pi_-(p)=q$ as being the first return map in forward
time of the flow of $X_-$ to $L_-,$ i.e. if $\varphi_-(s,p)$ is the solution of $\dot{\bf x}=A_-{\bf
x}+B_-$ such that $\varphi_-(0,p)=p$ and $p\in L_-^{O},$ then
$q=\varphi_-(s,p)\in L_-^{I}$ with $s\geq0$. Note that
$\Pi_-(p_-)=p_-.$

Now, using the parametrizations of $L_\pm^O$ and $L_\pm^I$
previously defined, given $p\in L_-^{O}$
and $q\in L_-^{I}$ there exist unique $c\geq0$ and $d\geq0$ such
that $p=p_--c\dot{p}_-$ and
$q=p_-+d\dot{p}_-.$ So $\Pi_-$ induces a mapping
$\pi_-$ given by $\pi_-({c})=d.$ Observe that the qualitative
behavior of $\Pi_-$ is equivalent to the qualitative behavior of
$\pi_-.$

Analogously we can define the Poincar\'{e} maps $\Pi_+$ from
$L_+^I$ to $L_+^O$, through the flow of $X_+$, $\Pi_o$ from
$L_-^I$ to $L_+^I$ and $\bar{\Pi}_o$ from $L_+^O$ to $L_-^O$,
bouth through of the flow of $X_o$ and, by the parametrization
defined to this transversal sections, we obtain the correspondents
induced   Poincar\'{e} maps $\pi_+$, $\pi_o$ and $\bar{\pi}_o$ as
above.

The maps  $\bar{\pi}_o$, $\pi_+$, $\pi_o$, and $\pi_-$ are invariant under linear change of coordinates
and translation, see  Proposition 4.3.7 in \cite{Llibre1}. Then, in the Case $b_2<-1$, to compute the maps $\pi_\pm$ we can suppose that the equilibria of $X_\pm$ are at the origin and that the matrices $A_\pm$ are given in their real Jordan normal form. 

In what follows, we will consider the maps $\bar{\pi}_o$, $\pi_+$, $\pi_o$, and $\pi_-$ instead of $\bar{\Pi}_o$, $\Pi_+$, $\Pi_o$, and $\Pi_-$. We will study the qualitative behavior of each one of these maps separately in order to understand the global behavior of the general Poincar\'{e} return map. 

The next lemma will be useful in the study of the Poincar\'{e}
return map associated to system (\ref{base system nf}). The proof
of this lemma can be found in \cite{Llibre1} Lemma 4.4.10 and also
in \cite{Mauricio}.

\begin{lemma}
Consider the function $\varphi:\mathbb{R}^2\rightarrow\mathbb{R}$
given by $\varphi(x,y)=1-e^{xy}(\cos x-y\sin x).$ The qualitative
behavior of $\varphi(x,y_o)$ in $(-\infty,2\pi)$ is represented in
Figure \ref{grafico1}  when $y_o>0,$ and in Figure \ref{grafico2}
when $y_o<0.$ \label{basicfunction}
\end{lemma}

\begin{figure}[!htb]
\begin{minipage}[b]{0.45\linewidth}
\centering
\includegraphics[width=4cm,height=5cm,angle=-90]{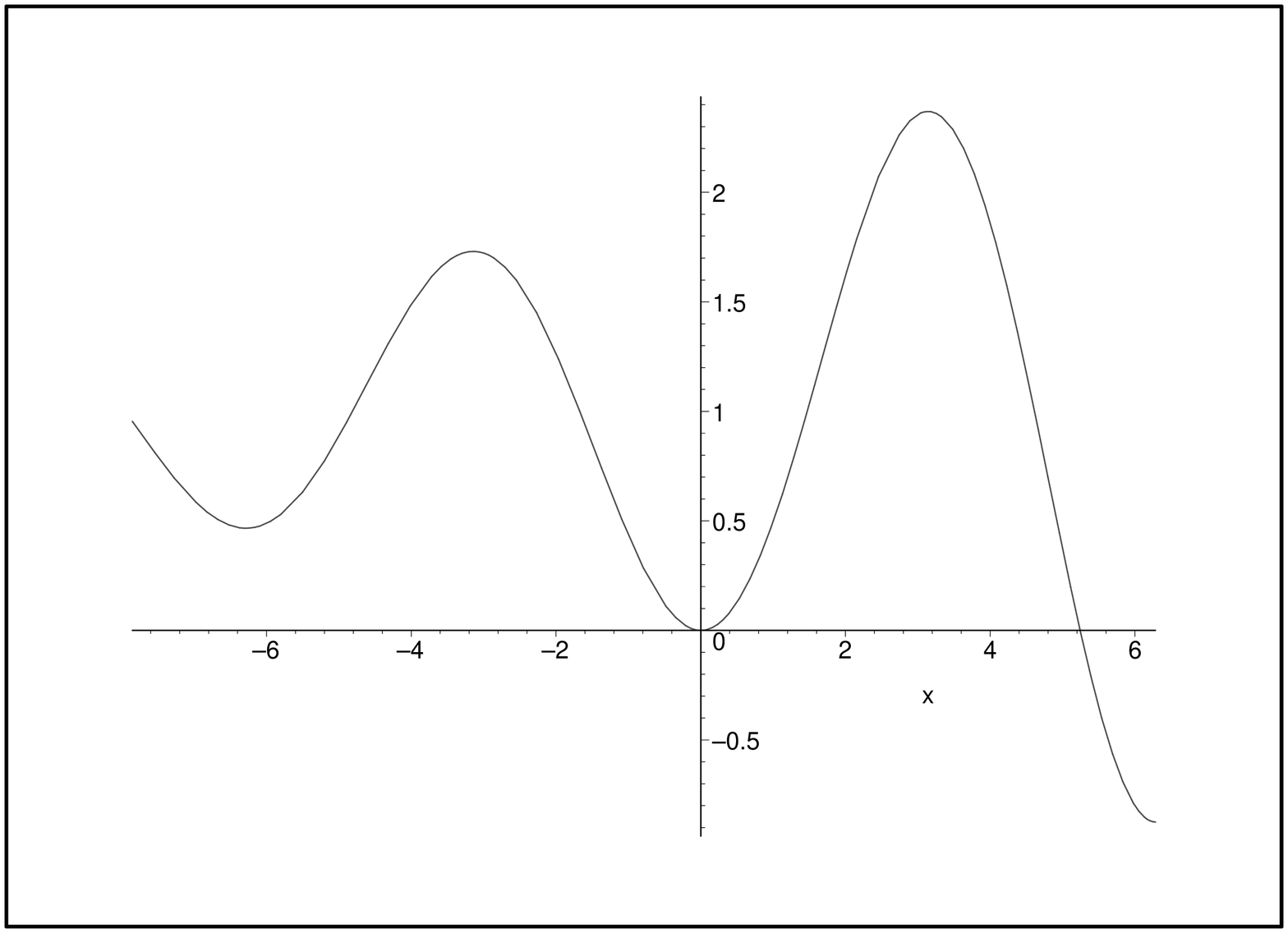}
\caption{Function $\varphi(x,y_o)$ for $y_o>0$.} \label{grafico1}
\end{minipage}
\hspace{0.5cm}
\begin{minipage}[b]{0.45\linewidth}
\centering
\includegraphics[width=4cm,height=5cm,angle=-90]{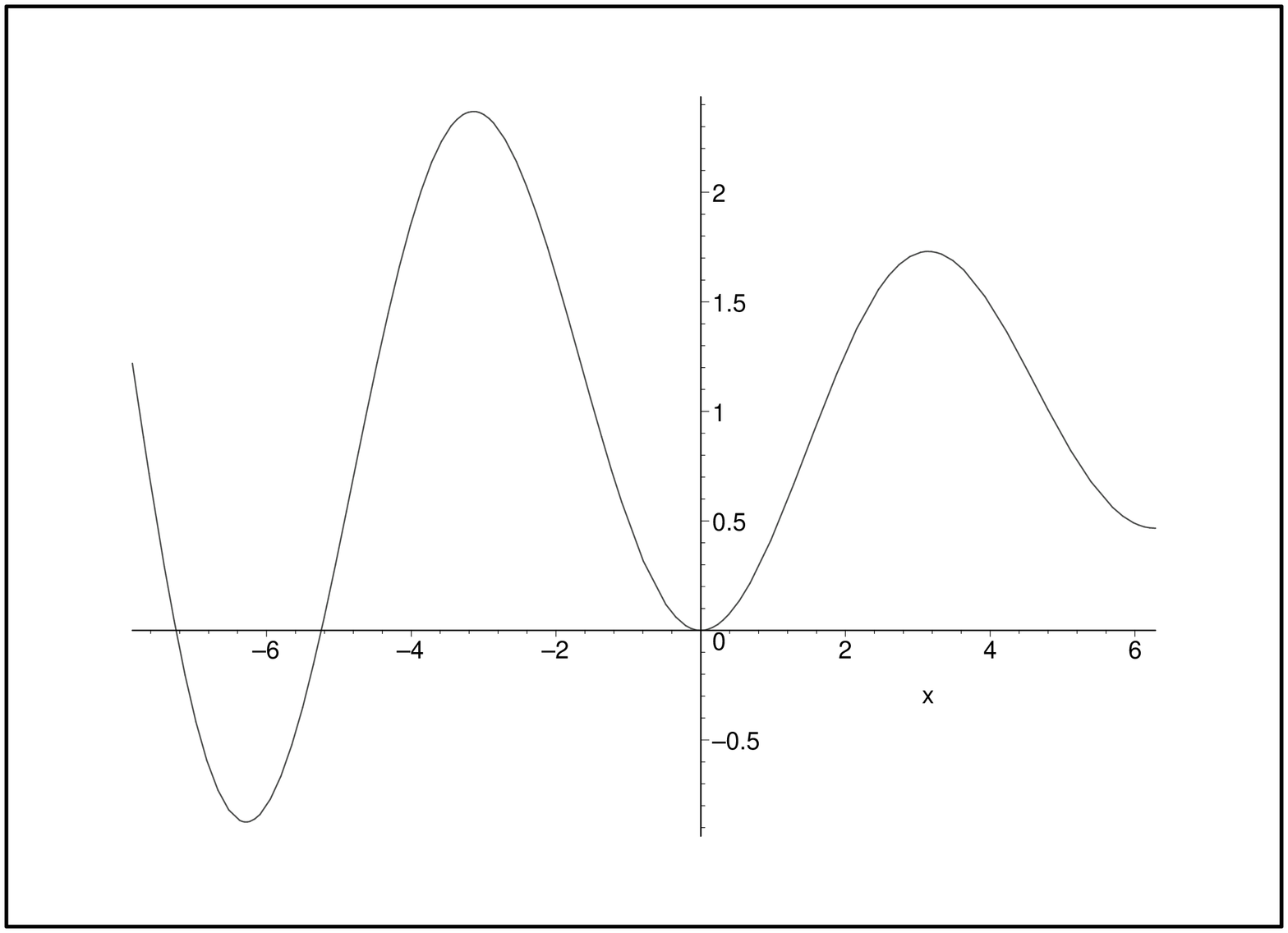}
\caption{Function $\varphi(x,y_o)$ for $y_o<0$.} \label{grafico2}
\end{minipage}
\end{figure}

\begin{proposition}\label{prop mau 1}
Consider the vector field $X_-$ in $R_-$ with a virtual center or  focus
equilibrium and such that $t_- \geq 0.$ Let $\pi_-$ be the map
associated to the Poincar\'{e} map $\Pi_-:L_-^O\rightarrow L_-^I$
defined by the flow of the linear system $\dot{{\bf x}}=A_-{\bf
x}+B_-.$
\begin{itemize}
\item[a] If $t_->0$, then the map $\pi_{-}$ is such that
$\pi_{-}:[0,\infty)\rightarrow[0,\infty),$ $\pi_{-}(0)=0,$
$\displaystyle\lim_{c\rightarrow\infty}\pi_{-}({c})=+\infty$ and
$\pi_{-}(c)>c$ in $(0,\infty).$
\begin{itemize}
\item[(a.1)] If $c\in(0,\infty)$, then
$(\pi_{-})'(c)=\dfrac{c}{\pi_-(c)}e^{2\gamma_-\tau_-}.$ Moreover
$(\pi_{-})'(c)>1$ and
$\displaystyle\lim_{c\rightarrow0}(\pi_-)'(c)=1.$

\item[(a.2)] If $c\in(0,\infty)$, then $(\pi_{-})''(c)>0.$

\item[(a.3)] The straight line
$d=e^{\gamma_-\pi}c-t_-(1+e^{\gamma_-\pi})/d_-$ in the plane
$(c,d)$ is an asymptote of the graph of $\pi_-$ when $c$ tends to
$+\infty$ where $\gamma_-=t_-/\sqrt{4d_--t_-^2}.$ So
$\displaystyle\lim_{c\rightarrow\infty}(\pi_-)'(c)=e^{\gamma_-\pi}.$
\end{itemize}

\item[(b)] If $t_-=0$, then $\pi_{-}$ is the identity map in $[0,\infty).$
 \end{itemize}
\end{proposition}

\begin{proof}
See Proposition 3.2 of \cite{Mauricio}.
\end{proof}

%
%
%



\begin{proposition}\label{prop.8}
Consider the vector field $X_+$ in $R_+$ with a real focus
equilibrium and such that $t_+>0.$ Let $\pi_+$ be the map
associated to the Poincar\'{e} map $\Pi_+:L_+^I\rightarrow L_+^O$
defined by the flow of the linear system $\dot{{\bf x}}=A_+{\bf
x}+B_+.$
\begin{itemize}
\item[(a)] Then the maps $\pi_{+}$ is such that
$\pi_{+}:[0,\infty)\rightarrow[b^\ast_+,\infty),$
$\pi_+(0)=b^\ast_+
>0$, $\displaystyle\lim_{a\rightarrow\infty}\pi_{+}(a)=+\infty$
and $\pi_{+}(a)>a$ in $(0,\infty).$
\begin{itemize}
\item[(a.1)] If $a\in(0,\infty)$, then
$(\pi_{+})'(a)=\dfrac{a}{\pi_+(a)}e^{2\gamma_+\tau_+}.$ Moreover
$(\pi_{+})'(a)>0$ and
$\displaystyle\lim_{a\rightarrow0}(\pi_+)'(a)=0.$

\item[(a.2)] If $a\in(0,\infty)$, then $(\pi_{+})''(a)>0.$

\item[(a.3)] The straight line
$b=e^{\gamma_+\pi}a+t_+(1+e^{\gamma_+\pi})/d_+$ in the plane
$(a,b)$ is an asymptote of the graph of $\pi_+$ when $a$ tends to
$+\infty$ where $\gamma_+=t_+/\sqrt{4d_+-t_+^2}.$ So
$\displaystyle\lim_{a\rightarrow\infty}(\pi_+)'(a)=e^{\gamma_+\pi}.$
\end{itemize}

\end{itemize}
\label{prop p1}
\end{proposition}

\begin{proof}
Let $p_+$ be the contact point of the flow with $L_+$ and $p$ and
$q$ such that $\Pi_+(p)=q$. As $q$ is in the orbit of $p$ in the
forward time we have that $q=\varphi(s,p)$ with $s\geq0.$ Moreover
for computing the map $\pi_+$ we can suppose that the real
equilibrium is at the origin and that matrix $A_+$ is in its real
Jordan normal form.

Let $p_+^*$ be the contact point $p_+$ in the coordinates in which
$A_+$ is in its real Jordan normal form and the real equilibrium
of $X_+$ is at the origin. We denote by $\dot{p}_+^*=X_+(p_+^*).$
So we can write
$$q=\varphi_+(s,p)=e^{A_+s}p.$$ As $p=p_+^*+a\dot{p}_+^*$ and
$q=p_+^*-\pi_+(a)\dot{p}_+^*$, we obtain
$$p_+^*-\pi_+(a)\dot{p}_+^*=e^{A_+s}(p_+^*+a\dot{p}_+^*).$$ Now using the fact
that $\dot{p}_+^*=A_+p_+^*$, we have
\begin{equation}
(Id-\pi_+(a)A_+)p_+^*=e^{A_+s}(Id+aA_+)p_+^*,\label{equation pi-}
\end{equation}
where $a\geq0,$ $\pi_+(a)\geq0,$ $s\geq0$ and the matrix $A_+$ is
given by $A_+=\left(\begin{array}{cc}
             \alpha_+ & -\beta_+ \\
             \beta_+ & \alpha_+
           \end{array}\right)$ with $\alpha_+=\dfrac{t_+}{2}.$ Since $p_+^*\neq(0,0)$,
we obtain from equation (\ref{equation pi-}), that $b=\pi_+(a)$ is
defined by the system
\begin{equation}
\begin{array}{l}
    \!\!\!\!1-b\alpha_+\!\!=e^{\alpha_+s}(\cos(\beta_+s)+a[\alpha_+\cos(\beta_+s)-\beta_+\sin(\beta_+s)]), \\
    \\
    \!\!\!\!b\beta_+\!\!= -e^{\alpha_+s}(\sin(\beta_+s)+a[\alpha_+\sin(\beta_+s)+\beta_+\cos(\beta_+s)]),
  \end{array}\label{system sol}
\end{equation}
and the inequalities $a\geq0,\,b\geq0$ and $s\geq0.$

Let $\pi_+(0)=b^\ast_+$, we have that $b^\ast_+ >0$ (see Figure
\ref{figg1}).

\begin{figure}[h!]
\begin{center}
\begin{overpic}[width=5cm,height=4.5cm]{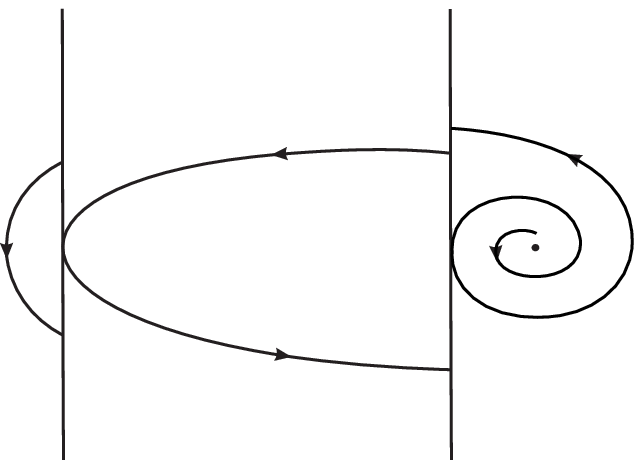}
\put(58,71){\tiny $b_+^*$} \put(6,-6){\tiny $L_-$}
\put(64,-6){\tiny $L_+$}
\end{overpic}
\end{center}
\caption{The flow of the three zones vector fields with a real
focus equilibrium and $t_+>0$.} \label{figg1}
\end{figure}

Moreover if $a=a_o,\,b=b_o$ and $s=s_o$ is a solution of
(\ref{system sol}), then $s_o$ is the flight time between the
points $p=p_+^*+a\dot{p}_+^*$ and $q=p_+^*-b\dot{p}_+^*.$

Define $\tau_+=\beta_+s$ and $\gamma_+=\alpha_+/\beta_+.$ Solving
system (\ref{system sol}) with respect to $\tau_+$ we obtain the
following parametric equations for $\pi_+(a)=b,$

\begin{equation}
\begin{array}{rr}
  a(\tau_+)=-\dfrac{\beta_+e^{-\gamma_+\tau_+}}{d_+\sin\tau_+}\varphi(\tau_+,\gamma_+) & \mbox{   and   } \\
  & \\
  b(\tau_+)=-\dfrac{\beta_+e^{\gamma_+\tau_+}}{d_+\sin\tau_+}\varphi(\tau_+,-\gamma_+), &  \\
\end{array}\label{parametricsolution}
\end{equation}
where $\varphi$ is the function described in Lemma
\ref{basicfunction}. Since $A_+$ is given in its real Jordan
normal form, $\tau_+$ is the angle covered by the solution during
the flight time $s$. Hence we conclude that $\tau_+ \in
(\pi,\tau^\ast)$, where $\tau^\ast<2\pi$. Note that $\tau^\ast$ is
the angle covered by the solution during the flight time $s^\ast$,
i.e. $e^{s^\ast A}p_+^\ast=\Pi_+(p_+^\ast)$.

 Now since
$\displaystyle\lim_{\tau_+\rightarrow\pi^+}a(\tau_+)=+\infty$ and
$\displaystyle\lim_{\tau_+\rightarrow\pi^+}b(\tau_+)=+\infty$ it
follows that the domain of definition of $\pi_+$ is $[0,+\infty)$
and $\displaystyle\lim_{a\rightarrow\infty}\pi_+(a)=+\infty.$

Moreover, when $\tau_+\in(\pi,\tau^\ast)$ we have
$$b(\tau_+)-a(\tau_+)=-\dfrac{2\beta_+}{d_+\sin\tau_+}(\sinh(\gamma_+\tau_+)-\gamma_+\sin\tau_+).$$
Since $\sinh(\gamma_+\tau_+)>\gamma_+\sin\tau_+$ when
$\tau_+\in(\pi,\tau^\ast),$ we conclude from the expression above
that $b(\tau_+)>a(\tau_+)$ if $\tau_+\in(\pi,\tau^\ast).$
Therefore $\pi_+(a)>a$ in $(0,+\infty).$ So statement (a) is
proved.

Derivating (\ref{parametricsolution}) with respect to $\tau_+$ it
follows that

$$
  \dfrac{da}{d\tau_+}=-\dfrac{\beta_+}{d_+\sin^2\tau_+}\varphi(\tau_+,-\gamma_+) $$
and
$$
\dfrac{db}{d\tau_+}=-\dfrac{\beta_+}{d_+\sin^2\tau_+}\varphi(\tau_+,\gamma_+).$$

Thus
$(\pi_+)'(a)=\dfrac{\varphi(\tau_+,\gamma_+)}{\varphi(\tau_+,-\gamma_+)}=\dfrac{a}{b}e^{2\gamma_+\tau_+}$
and $\displaystyle\lim_{a\rightarrow0}(\pi_+)'(a)=0$, because
$\displaystyle\lim_{a\rightarrow0}
b=\displaystyle\lim_{a\rightarrow0} \pi_+(a)=b^\ast_+\neq 0$.
Therefore substatement (a.1) is proved.

Now we observe that
$$\begin{array}{c}
  (\pi_+)''(a)=\dfrac{d}{d\tau_+}\left(\dfrac{db}{da}\right)\dfrac{1}{\frac{da}{d\tau_+}}= \\
  - \dfrac{2d_+(1+\gamma_+^2)\sin^3\tau_+}{\beta_+\varphi(\tau_+,-\gamma_+)^3}(\sinh(\gamma_+\tau_+)-\gamma_+\sin\tau_+)>0. \\
\end{array}$$

Therefore substatement (a.2) follows.

From expression (\ref{parametricsolution}) it follows that

$$
 \lim_{a\rightarrow\infty}\dfrac{\pi_+(a)}{a} =\displaystyle\lim_{\tau_+\rightarrow\pi}\dfrac{b(\tau_+)}{a(\tau_+)}=
\lim_{\tau_+\rightarrow\pi}e^{2\gamma_+\tau_+}\dfrac{\varphi(\tau_+,-\gamma_+)}{\varphi(\tau_+,\gamma_+)}
=e^{\gamma_+\pi}.
$$

 On the other hand by applying the L'H\^{o}ptal's rule it is easy
to check that
$\displaystyle\lim_{a\rightarrow+\infty}(\pi_+(a)-e^{\gamma_+\pi}a)=t_+(1+e^{\gamma_+
\pi})/d_+$, which implies that the straight line
$b=e^{\gamma_+\pi}a+t_+(1+e^{\gamma_+ \pi})/d_+$ is an asymptote of
the graph of $\pi_+(a)$ and we obtain substatement (a.3).

\end{proof}

\begin{proposition}\label{prop.9}
Consider the vector field $X_+$ in $R_+$ with a real focus
equilibrium and such that $t_+<0.$ Let $\pi_+$ be the map
associated to the Poincar\'{e} map $\Pi_+:D_+^*\subset
L_+^I\rightarrow L_+^O$ defined by the flow of the linear system
$\dot{{\bf x}}=A_+{\bf x}+B_+$, where $D_+^*$ is a subset of
$L_+^I$ which the mapping $\Pi_+$ is well defined.
\begin{itemize}
\item[(a)] Then the maps $\pi_{+}$ is such that
$\pi_{+}:[a^\ast_+,\infty)\rightarrow[0,\infty),$
$\pi_+(a^\ast_+)=0$,
$\displaystyle\lim_{a\rightarrow\infty}\pi_{+}(a)=+\infty$ and
$\pi_{+}(a)<a$ in $(a^\ast_+,\infty).$
\begin{itemize}
\item[(a.1)] If $a\in(a^\ast_+,\infty)$, then
$(\pi_{+})'(a)=\dfrac{a}{\pi_+(a)}e^{2\gamma_+\tau_+}.$ Moreover
$(\pi_{+})'(a)>0$ and $\displaystyle\lim_{a\rightarrow
a^\ast_+}(\pi_+)'(a)=+\infty.$

\item[(a.2)] If $a\in(a^\ast_+,\infty)$, then $(\pi_{+})''(a)<0.$

\item[(a.3)] The straight line
$b=e^{\gamma_+\pi}a+t_+(1+e^{\gamma_+\pi})/d_+$ in the plane
$(a,b)$ is an asymptote of the graph of $\pi_+$ when $a$ tends to
$+\infty$ where $\gamma_+=t_+/\sqrt{4d_+-t_+^2}.$ So
$\displaystyle\lim_{a\rightarrow\infty}(\pi_+)'(a)=e^{\gamma_+\pi}.$
\end{itemize}

\end{itemize}\label{prop p1}
\end{proposition}

%
%
%
%

\begin{proof}
The proof follows in a similar way to the proof of Proposition
\ref{prop.8}.
\end{proof}

\begin{proposition}\label{prop.10}
Consider the vector field $X_o$ in $R_o$ with a virtual focus
equilibrium. Let $\pi_o$ be the map associated to the Poincar\'{e}
map $\Pi_o: L_-^I\rightarrow L_+^I$ defined by the flow of
the linear system $\dot{{\bf x}}=A_o{\bf x}+B_o$ from the straight line $L_-$ to the straight
line $L_+$.
\begin{itemize}
\item[(a)] Then the map $\pi_o$ is such that
$\pi_o:[0,\infty)\rightarrow[a^*_o,\infty),$ $a^*_o> 0$ with
$\pi_o(0)=a^*_o$ and
$\displaystyle\lim_{d\rightarrow\infty}\pi_o(d)=+\infty.$

\item[(b)] If
$d\in[0,\infty)$, then
$\pi_o'(d)=\left(\dfrac{1-b_2}{1+b_2}\right)^2
e^{2\gamma_o\tau_o}\dfrac{d}{\pi_o(d)},$ with $\tau_o\rightarrow0$
when $d\rightarrow\infty$ and
$\displaystyle\lim_{b\rightarrow\infty}\pi_o'(d)=-\dfrac{1-b_2}{1+b2}$.
\end{itemize}
\end{proposition}

\begin{proof}
See Proposition 3.3 of \cite{ClauWeMau}.
\end{proof}

\begin{proposition} \label{prop.11} Consider the vector field $X_o$ in $R_o$ with a
virtual focus equilibrium. Let $\bar{\pi}_o$ be the map associated
to the Poincar\'{e} map $\bar{\Pi}_o:\bar{D}_o^*\subset
L_+^O\rightarrow L_-^O$ defined by the flow of the linear system
$\dot{{\bf x}}=A_o{\bf x}+B_o$ from the straight line $L_+$ to the straight
line $L_-$, where $\bar{D}_o^*$ is a subset of $L_+^O$ in which the mapping $\bar{\Pi}_o$ is well defined.
\begin{itemize}
\item[(a)] Then the map $\bar{\pi}_o$ is such that
$\bar{\pi}_o:[b^*_o,\infty)\rightarrow[0,\infty),$ $b^*_o> 0$ with
$\bar{\pi}_o(b^*_o)=0$ and
$\displaystyle\lim_{b\rightarrow\infty}\bar{\pi}_o(b)=+\infty.$

\item[(b)] If
$b\in(b^*_o,\infty)$, then
$\bar{\pi}_o'(b)=\left(\dfrac{b_2+1}{b_2-1}\right)^2
e^{2\gamma_o\bar{\tau}_o}\dfrac{b}{\bar{\pi}_o(b)},$ with
$\bar{\tau}_o\rightarrow0$ when $b\rightarrow\infty$,
$\displaystyle\lim_{b\rightarrow\infty}\bar{\pi}_o'(b)=\dfrac{b_2+1}{b_2-1}$
and $\displaystyle\lim_{b\rightarrow
b^*_o}\bar{\pi}_o'(b)=\infty$.
\end{itemize}
\end{proposition}

\begin{proof}
See Proposition 3.3 of \cite{ClauWeMau}.
\end{proof}

\section{Limit Cycles when  $b_2<-1$ and $X_+$ Has a Real Focus}
\label{sec Lim b_2< -1}

In this section we will assume that $b_2\leq-1$ and $X_+$ has a real focus  and
so, by Lemma \ref{equilibrium}\;{(i)-(ii)} and hypothesis (H1) and
(H2), $X_-$ has a virtual center and $X_o$ has either a virtual focus  (when $b_2<-1$) or a real focus at $(1,0)$ (when $b_2=-1$). As the previuous sections, we denote by 	$\gamma_i=\dfrac{\alpha_i}{\beta_i},\,i\in\{-,o,+\}.$ Thus
$\gamma_-=0$ and $\gamma_+\gamma_o<0.$ 

By proof of Proposition 12 from \cite{ClauWeMau}, if $b_2=-1$ and $\gamma_++\gamma_o=0$ the foci of $X_+$ and $X_o$ give rise to a center at the point $(1,0)$. Moreover, we have a bounded period annulus which the border is the equilibriun point $(1,0)$ and the periodic orbit tangent to $L_-$ at point $(-1,0)$, see Figure \ref{figura1a}.

\begin{figure}[h!]
	\begin{center}
		\begin{overpic}[width=5.5cm,height=3.5cm]{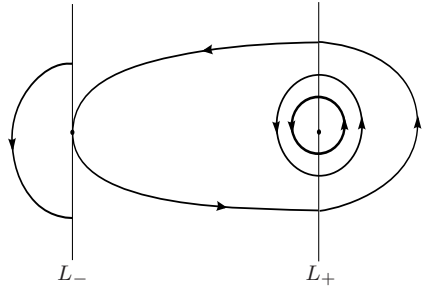}
				\put(12,-4){\tiny$L_-$} \put(72,-4){\tiny $L_+$}
		\end{overpic}
	\end{center}
	\caption{Period annulus which the border  is the equilibrium point $(1,0)$ and the orbit by the point $(-1,0)$. } \label{figura1a}
\end{figure}

Note that for $b_2<-1$  with $|b_2+1|$ small enough, the focus of $X_+$ is real, $a_+^*<a_o^*$ and $a_+^*<b_o^*$,  where $a_o^*=\pi_o(0)$, $b_o^*=\bar\pi_o^{-1}(0)$ and $a_+^*=\pi_+(0)$. Hence, the orbits of the periodic annuls are broken and give us the four possible phase portraits described in the Figures \ref{figura8ab} and \ref{figura9ab}, with $a_+^o=\pi_+^{-1}(b_o^*)$.

\begin{figure}[!htb]
	\begin{minipage}[b]{0.45\linewidth}
		\centering
		\begin{overpic}[width=5.5cm,height=3.5cm]{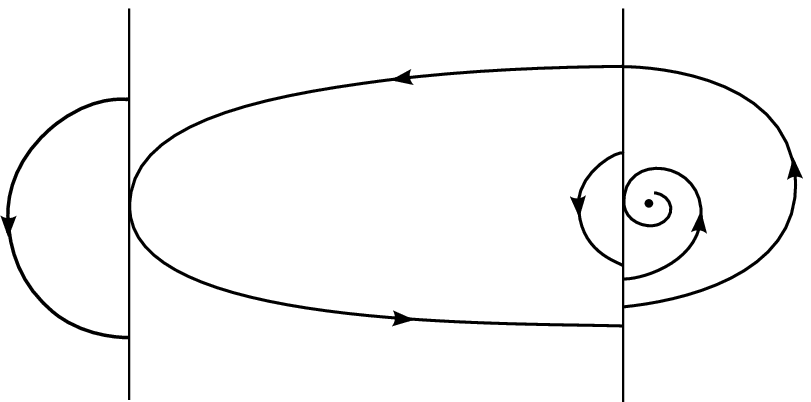}
\put(17,30){\tiny $0$} 
\put(78,12){\tiny $a_o^*$} 
\put(71,15){\tiny $a_o^+$} 
\put(69,21){\tiny $a_+^*$} 
\put(78,56){\tiny $b_o^*$}
\put(13,-3){\tiny $L_-$} 
\put(74,-3){\tiny $L_+$}
\put(38,-15){ Case $(a)$}
	\end{overpic}
	\end{minipage}
	\hspace{0.5cm}
	\begin{minipage}[b]{0.45\linewidth}
		\centering
	\begin{overpic}[width=5.5cm,height=3.5cm]{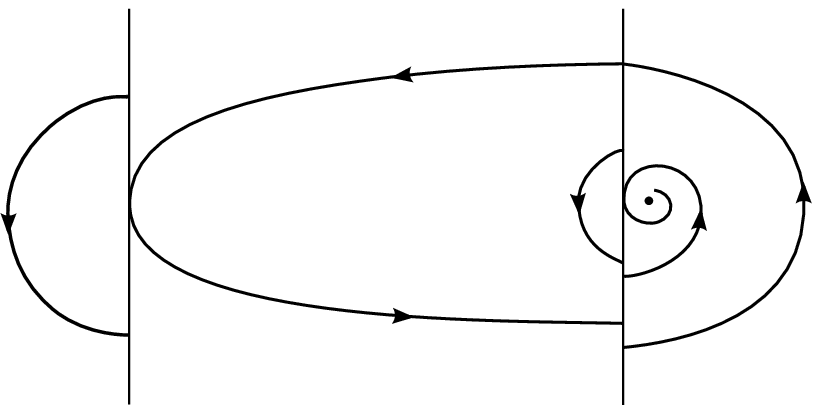}
\put(16.5,31){\tiny $0$} 
\put(77.5,13){\tiny $a_o^*$} 
\put(70,7){\tiny $a_o^+$} 
\put(70,20){\tiny $a_+^*$} 
\put(77,56){\tiny $b_o^*$}
\put(13,-3){\tiny $L_-$} 
\put(74,-3){\tiny $L_+$}
\put(38,-15){ Case $(b)$}
	\end{overpic}
	\end{minipage}
	\vspace{1cm}
		\caption{Phase portraits when the periodic annulus is broken and $\gamma_+<0$.} \label{figura8ab}
\end{figure}

\begin{figure}[!htb]
	\begin{minipage}[b]{0.45\linewidth}
		\centering
		\begin{overpic}[width=5.5cm,height=3.5cm]{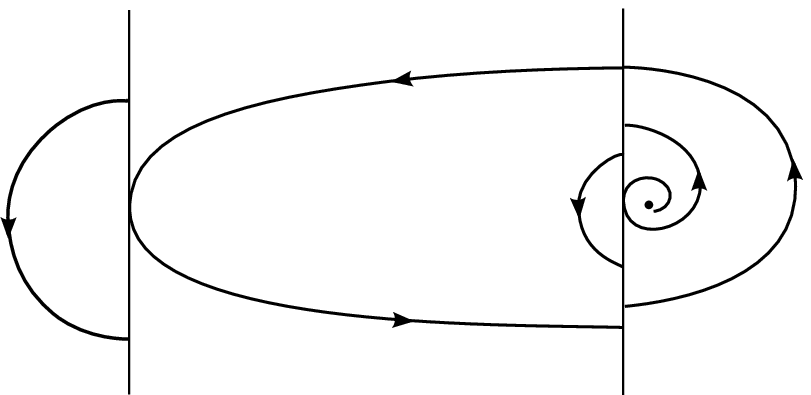}
\put(16.5,30){\tiny $0$} 
\put(78,10.5){\tiny $a_o^*$} 
\put(71.5,14.5){\tiny $a_o^+$} 
\put(71.5,43.5){\tiny $a_+^*$} 
\put(78,56){\tiny $b_o^*$}
\put(13,-3){\tiny $L_-$} 
\put(74,-3){\tiny $L_+$}
\put(38,-15){ Case $(a)$}
		\end{overpic}
	\end{minipage}
	\hspace{0.5cm}
	\begin{minipage}[b]{0.45\linewidth}
		\centering
		\begin{overpic}[width=5.5cm,height=3.5cm]{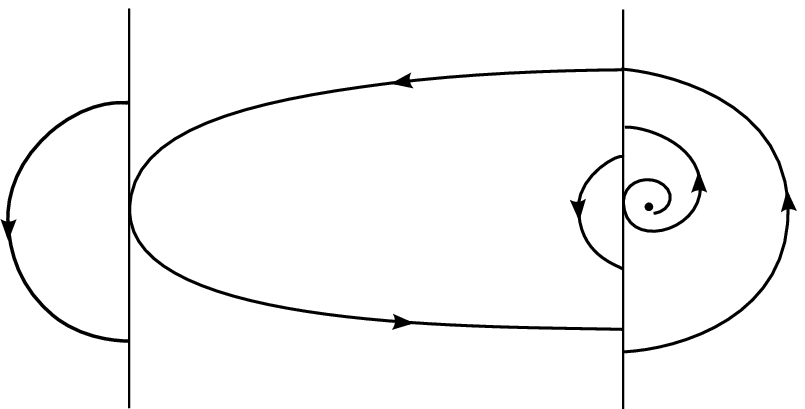}
\put(16.5,31){\tiny $0$} 
\put(78.5,12.5){\tiny $a_o^*$} 
\put(71,7){\tiny $a_o^+$} 
\put(72,44){\tiny $a_+^*$} 
\put(79,56){\tiny $b_o^*$}
\put(13,-3){\tiny $L_-$} 
\put(75,-3){\tiny $L_+$}
\put(38,-15){ Case $(b)$}
		\end{overpic}
	\end{minipage}
	\vspace{1cm}
	\caption{Phase portraits when the periodic annulus is broken and $\gamma_+>0$.} \label{figura9ab}
\end{figure}

We will study the periodic orbits of period annulus that persist by small perturbations of the parameters of system {\rm (\ref{base system nf})}. More precisely, we are interested in the limit cycles that bifurcate from the period annulus. For this, we will consider the Poincar\'e maps defined in two and three zones  respectively.  

Firstly, we begin studying the sign of the displacement function 
\[
a_o^*-a_o^+.
\]
Note that, for $b=-1$ and $\gamma_++\gamma_o=0$, $a_o^*-a_o^+=0$. 

\begin{lemma} Let $\varphi_o(s, q)=(\varphi^1_o(s, q), \varphi^2_o(s, q))$ be the flow of $X_o$ by the point $q$. Then $\Pi_o(0) =  \varphi^2_o(s_o, (-1, 0))$ and $\bar{\Pi}_o^{-1}(0) =  \varphi^2_o(-\bar{s}_o, (-1, 0))$, where $s_o$ and $\bar{s}_o$  satisfy $\varphi^1_o( s_o, (-1, 0))=1$ and $\varphi^1_o(-\bar{s}_o, (-1, 0))=1$, respectively. Moreover $s_o$ and $\bar{s}_o$ are functions of $b_2$ and satisfy
\begin{eqnarray}
\Pi_o(0) & = & -2 e^{\gamma_o\tau_o}+ \left( -2\alpha_o+e^{\gamma_o\tau_o} \right) 
 \left(b_2+1\right) +\dfrac{1}{4} e^{-\gamma_o\tau_o } \left( b_2
+1 \right) ^{2}  \label{eq:9} \\ &+ & \mathcal{O}\left((b2+1)^3\right),\nonumber \\
\bar{\Pi}_o^{-1}(0) & = & 2 e^{-\gamma_o\bar{\tau}_o}- \left( 2\alpha_o+e^{-\gamma_o\bar{\tau}_o} \right) 
 \left(b_2+1\right) -\dfrac{1}{4} e^{\gamma_o\bar{\tau}_o } \left( b_2
+1 \right) ^{2}  \label{eq:8}\\ &+ & \mathcal{O}\left((b2+1)^3\right),\nonumber
 \end{eqnarray}
 where $\beta_os_o(-1)=\tau_o=\arctan \left(\dfrac{\beta_o}{\alpha_o}\right)$ and $\beta_o\bar{s}_o(-1)=\bar{\tau}_o=\arctan \left(-\dfrac{\beta_o}{\alpha_o}\right)$, i.e.  for $b_2=-1$, $\varphi^1_o( \tau_o/\beta_o, (-1, 0))=1$ and $\varphi^1_o(-\bar{\tau}_o/\beta_o, (-1, 0))=1$, respectively.
\end{lemma}
\begin{proof}
The first statement follows directly from the definitions of $\Pi_o(0)$ and $\bar{\Pi}_o^{-1}(0)$.

We have that
\begin{equation}
\label{eq:10}
\begin{aligned}
\varphi_o^1 (s, (-1,0))  & =   (b_2-1) e^{\alpha_o s}\cos (\beta_o s)-\gamma_o(b_2-1) e^{\alpha_o s} \sin(\beta_o s)-b_2,  \\
\varphi_o^2 (s, (-1, 0)) &  =  (b2-1)\dfrac{e^{\alpha_o s}}{\beta_o}\sin (\beta_o s). 
\end{aligned}
\end{equation}
The equation $\varphi^1_o(s, (-1, 0))=1$ can be written as
\begin{equation}
\label{eq:11}
(b_2-1)e^{\alpha_o s}\left[\cos (\beta_o s)-\gamma_o \sin (\beta_o s)\right]=b_2+1.
\end{equation}
Hence, for $b_2=-1$, this equation is equivalent to
\[
\cos (\beta_o s)-\gamma_o \sin (\beta_o s)=0,
\]
whose solution $s_o(-1)$ is
\[
s_o(-1)=\dfrac{1}{\beta_o}\arctan (\gamma_o^{-1})=\dfrac{\tau_o}{\beta_o},
\]
i.e. $\tau_o=\arctan \left(\dfrac{\beta_o}{\alpha_o}\right)$. Note that, if $\alpha_o>0$, then $0<\tau_o<\dfrac{\pi}{2}$. Now, if $\alpha_o<0$, then $\dfrac{\pi}{2}<\tau_o<\pi$. Moreover it follows from equation \eqref{base system nf} that $\det A_o=1.$ This implies that $\alpha_o^2+\beta_o^2=1.$ So we have $\cos\tau_o=\alpha_o$ and $\sin\tau_o=\beta_o$. 

Computing the derivative with respect to $b_2$ in both sides of equation \eqref{eq:11}, treating $s$ as a function of $b_2$ and after substituting $b_2=-1$, we have
\[
s_o'(-1)=\dfrac{e^{-\alpha_os_o(-1)}}{2\beta_o(\sin(\tau_o)+\gamma_o\cos(\tau_o))}=\frac{1}{2}e^{-\gamma_o\tau_o}.
\]
Analogously, we obtain 
\[
s_o''(-1)=\dfrac{1}{2}\left(1-\alpha_o e^{-\gamma_o\tau_o}\right)e^{-\gamma_o\tau_o}.
\]
Thus, the second order expansion of the function $s_o(b_2)$, solution of the equation $\varphi^1_o(s, (-1, 0))=1$, in a neighborhood of $b_2=-1$ is
\[
s_o(b_2)=\dfrac{\tau_o}{\beta_o}+\frac{1}{2}e^{-\gamma_o\tau_o}(b_2+1)+\dfrac{1}{4}\left(1-\alpha_o e^{-\gamma_o\tau_o}\right)e^{-\gamma_o\tau_o}(b_2+1)^2+\mathcal{O} ((b_2+1)^3).
\]
Now, substituting the above expression in $\varphi_o^2 (s, (-1, 0)) $ and expanding it as a function of $b_2$ in a neighborhood of $b_2=-1$, we obtain the second order expansion of $\Pi_o(0)$, i.e.
\[
\Pi_o(0)  =  -2 e^{\gamma_o\tau_o}+ \left( -2\alpha_o+e^{\gamma_o\tau_o} \right) 
 \left(b_2+1\right) +\dfrac{1}{4} e^{-\gamma_o\tau_o } \left( b_2
+1 \right) ^{2} +\mathcal{O} \left((b2+1)^3\right).
\]

In a similar way, we can determine the second order expansion of the function $\bar{s}_o(b_2)$, solution of the equation $\varphi^1_o(-s, (-1, 0))=1$, in a neighborhood of $b_2=-1$. Just changing in \eqref{eq:10} $s$ by $-s$ and noting that in this case the equation  $\varphi^1_o(-s, (-1, 0))=1$, for $b_2=-1$, is equivalent to 
\[
\cos (\beta_o s)+\gamma_o \sin (\beta_o s)=0,
\]
whose solution $\bar{s}_o(-1)$ is
\[
\bar{s}_o(-1)=\dfrac{1}{\beta_o}\arctan (-\gamma_o^{-1})=\dfrac{\bar{\tau}_o}{\beta_o},
\]
i.e. $\bar{\tau}_o=\arctan \left(-\dfrac{\beta_o}{\alpha_o}\right)$. Moreover, if $\alpha_o>0$, then $\dfrac{\pi}{2}<\tau_o<\pi$ and if $\alpha_o<0$, then $0<\tau_o<\dfrac{\pi}{2}$. In both cases, as $\alpha_o^2+\beta_o^2=1$, we have $\cos\tau_o=-\alpha_o$ and $\sin\tau_o=\beta_o$.  Hence, we obtain
\[
\bar{s}_o(b_2)=\dfrac{\bar{\tau}_o}{\beta_o}+\frac{1}{2}e^{\gamma_o\bar{\tau}_o}(b_2+1)+\dfrac{1}{4}\left(1+\alpha_o e^{\gamma_o\bar{\tau}_o}\right)e^{\gamma_o\bar{\tau}_o}(b_2+1)^2+\mathcal{O}((b_2+1)^3),
\]
and so
\[
\bar{\Pi}_o^{-1}(0) =  2 e^{-\gamma_o\bar{\tau}_o}- \left( 2\alpha_o+e^{-\gamma_o\bar{\tau}_o} \right) 
 \left(b_2+1\right) -\dfrac{1}{4} e^{\gamma_o\bar{\tau}_o } \left( b_2
+1 \right) ^{2} +\mathcal{O}\left((b2+1)^3\right).
\]
\end{proof}

In what follows, we denote by $d_+^*$,  $\beta_+^*$ and $\gamma^*_+$ the restrictions  of $d_+$,  $\beta_+$ and $\gamma_+$ to $b_2=-1$, respectively.  

\begin{lemma} Let $\varphi_+(s, q)=(\varphi^1_+(s, q), \varphi^2_+(s, q))$ be the flow of $X_+$ by the point $q=(1, \bar{\Pi}_o^{-1}(0))$, where $\bar{\Pi}_o^{-1}(0)$ is given by \eqref{eq:8}. Then $\Pi_+^{-1}\circ \bar{\Pi}_o^{-1}(0) =\varphi^2_+(-s_+, (1, $ $\bar{\Pi}_o^{-1}(0)))$, where $s_+$ satisfy $\varphi^1_+(- s_+, (1, \bar{\Pi}_o^{-1}(0)))=1$. Moreover $s_+$ is a function of $b_2$ and
\begin{equation}
\label{eq:13}
\begin{aligned}
\Pi_+^{-1}\circ \bar{\Pi}_o^{-1}(0) = & -2 e^{-(\gamma_o\bar{\tau}_o+\gamma^*_+\pi)} \\ &+\left(e^{-(\gamma_o\bar{\tau}_o+\gamma_+^*\pi)}\left[2e^{\gamma_o\bar{\tau}_o}\left(\alpha_o-\dfrac{\alpha_+}{d_+^*}\right)+\left(1-\dfrac{\pi\gamma_+^*}{(\beta_+^*)^2}\right)\right]\right. \\ & \left.-\dfrac{2\alpha_+}{d_+^*}\right)(b_2+1) \\
& +\frac{e^{-(\gamma_o\bar{\tau}_o+\gamma^*_+\pi)} }{4( d_+^*)^2 (\beta_+^*) ^6}\left[d_+^* (\beta_+^*) ^6
   e^{2 \gamma_o\bar{\tau}_o} \left(d_+^*+e^{2\gamma_+^* \pi}-1\right) \right. \\ & \left.+4 \alpha_+  (\beta_+^*)^3 e^{\gamma_o\bar{\tau}_o} \left(\pi  d_+^* (\alpha_o
   d_+^*-\alpha_+ )\right.+2 (\beta_+^*)^3 \left(e^{\gamma_+^*\pi  \alpha}+1\right)\right)\\
   &+\left. \pi  (d_+^*)^2 \alpha_+  \left(2 (\beta_+^*)^3-\pi  \alpha_+ +3 \beta_+^* \right)\right] (b_2+1)^2\\
& +\mathcal{O} ((b_2+1)^3).
\end{aligned}
\end{equation}
\end{lemma}
\begin{proof}
The first statement follow directly from the definition of $\Pi_+^{-1}(0)$.

We have that
\[
\begin{aligned}
\varphi_+^1 (-s_+, (1, \bar{\Pi}_o^{-1}(0)))   = &  \dfrac{(1+b_2)}{d_+} e^{-\alpha_+ s}\cos (\beta_+ s)\\ &+\dfrac{d_+ \bar{\Pi}_o^{-1}(0)+(1+b_2)\alpha_+}{\beta_+d_+} e^{-\alpha_+ s} \sin(\beta_+ s)\\ &+\dfrac{d_+-(1+b_2)}{d_+},  \\
\varphi_+^2 (-s_+, (1, \bar{\Pi}_o^{-1}(0)))  = &  \dfrac{d_+ \bar{\Pi}_o^{-1}(0)+c_{11}(1+b_2)}{d_+} e^{-\alpha_+ s}\cos (\beta_+ s) \\ &-\left[\dfrac{2(1+b_2)(d_+-c_{11}t_+)}{2\beta_+d_+}\right. \\ & \left.+\dfrac{ d_+(a_1-c_{11})\bar{\Pi}_o^{-1}(0)}{2\beta_+d_+}\right] e^{-\alpha_+ s} \sin(\beta_+ s)\\ &-\dfrac{c_{11}(1+b_2)}{d_+}. 
\end{aligned}
\]
Note that, for $b_2=-1$, $s_+(-1)=\dfrac{\pi}{\beta_+^*}$ is the solution of the equation $\varphi_+^1 (-s_+, (1, $ $\bar{\Pi}_o^{-1}(0)))=1$. Thus, if $s_+(b_2)=\dfrac{\pi}{\beta_+^*}+s_{+_1}(b_2+1)+s_{+_2}(b_2+1)^2+\cdots$ is the solution of this equation, substituting $\bar{\Pi}_o^{-1}(0)$ by its expansion \eqref{eq:8} and expanding $\varphi_+^1 (-s_+, (1, \bar{\Pi}_o^{-1}(0)))$ as a power series in $(b_2+1)$ in a neighborhood of $b_2=-1$, we obtain 
\[
\begin{aligned}
s_+(b_2)= & \dfrac{\pi}{\beta_+^*}-\left( \dfrac{1+e^{-\gamma_+^*\pi}}{2d_+^*}\right)e^{\gamma_o\bar{\tau}_o+\gamma_+^*\pi}(b_2+1)  \\ &+
e^{\gamma_o\bar{\tau}_o} \left(\left(\dfrac{e^{\gamma_+^*\pi}+1}{2(d_+^*)^2}\right) \left(\frac{1}{2} \alpha_+  e^{\gamma_o\bar{\tau}_o}
   \left(e^{\gamma_+^*\pi}+1\right)+1\right)\right. \\ & +\left.\dfrac{e^{\gamma_+^*\pi} }{4 d_+^*}\left(-2 \alpha_o e^{\gamma_o\bar{\tau}_o}+\frac{\pi  \alpha_+ }{(\beta_+^*) ^3}-1\right)-\dfrac{2 \alpha_o e^{\gamma_o\bar{\tau}_o}+1}{4 d_+^*}\right)(b_2+1)^{2}+ \cdots.
\end{aligned}
\]
Hence, substituting the above expression in $\varphi_+^2(-s_+(b_2),(1,\bar{\Pi}_o^{-1}(0)))$ and expanding it as a power series in $(b_2+1)$ at the neighborhood of $b_2=-1$, we obtain the second order expansion of  $\Pi_+^{-1}\circ \bar{\Pi}_o^{-1}(0)$ in $(b_2+1)$ at the neighborhood of $b_2=-1$, i.e. we have the equality \eqref{eq:13} .
\end{proof}

\begin{lemma}
Let  $a_o^*=\pi_o(0)$ and $a_+^o=\pi_+^{-1}(b_o^*)$. Then the diference $a_o^*-a_o^+$ is equivalent to $(\Pi_o - \Pi_+^{-1}\circ \bar{\Pi}_o^{-1})(0)$ and 
\begin{equation}
\label{eq:14}
\begin{aligned}
(\Pi_o - \Pi_+^{-1}\circ \bar{\Pi}_o^{-1})(0)= & 2e^{-(\gamma_o\bar{\tau}_o+\gamma_+^*\pi)}\left[1-e^{(\gamma_o+\gamma_+^*)\pi}\right] \\ & +e^{-(\gamma_o+\gamma_+^*)\pi}\left[
e^{\gamma_o\tau_o}\left(\dfrac{\pi t_+}{2(\beta_+^*)^3}+e^{(\gamma_o+\gamma_+^*)\pi}-1\right) \right. \\ & +\left.\left(\dfrac{t_+}{d_+^*}-t_o\right)\left(e^{(\gamma_o+\gamma_+^*)\pi}+e^{\gamma_o\pi}\right)\right](b_2+1)+\mathcal{O}((b_2+1)^2),
\end{aligned}
\end{equation}
where $\Pi_o(0)$ and $ \Pi_+^{-1}\circ \bar{\Pi}_o^{-1}(0)$ are given by \eqref{eq:9} and \eqref{eq:13}, respectively. 
\end{lemma}
\begin{proof}
The first statement is a consequence of the definitions of $\Pi_o(0)$ and  \linebreak$\left(\Pi_+^{-1}\circ \bar{\Pi}_o^{-1}\right)(0)$.  Now, the equality \eqref{eq:14} follows directly from \eqref{eq:9} and \eqref{eq:13}, noting that 
\[
\tau_o+\bar{\tau}_o=\arctan \left(\dfrac{\beta_o}{\alpha_o}\right)+\arctan \left(-\dfrac{\beta_o}{\alpha_o}\right)=\pi.
\]
\end{proof}

\begin{proposition}\label{prop.9.2}
For $|b_2+1|$ small enough, we have:
\begin{itemize}
\item[(a)] if $\gamma_o+\gamma_+^*>0$, then $a_o^*-a_o^+>0$;
\item[(b)] if $\gamma_o+\gamma_+^*<0$, then $a_o^*-a_o^+<0$;
\item[({c})] if $\gamma_o+\gamma_+^*=0$, $\gamma_o>0$ and $b_2<-1$ (resp. $b_2>-1$), then $a_o^*-a_o^+<0$ (resp. $a_o^*-a_o^+>0$);
\item[(d)] if $\gamma_o+\gamma_+^*=0$, $\gamma_o<0$ and $b_2<-1$ (resp. $b_2>-1$), then $a_o^*-a_o^+>0$ (resp. $a_o^*-a_o^+<0$).
\end{itemize}
\end{proposition}
\begin{proof}
Note that, when $(\Pi_o - \Pi_+^{-1}\circ \bar{\Pi}_o^{-1})(0)>0$ (resp. $(\Pi_o - \Pi_+^{-1}\circ \bar{\Pi}_o^{-1})(0)<0$), then $a_o^*-a_o^+<0$ (resp. $a_o^*-a_o^+>0$). Hence, by \eqref{eq:14}, if $\gamma_o+\gamma_+^*\neq 0$ the sign of $(\Pi_o - \Pi_+^{-1}\circ \bar{\Pi}_o^{-1})(0)$ is determined by sign of
\[
1-e^{(\gamma_o+\gamma_+^*)\pi},
\]
and the statement (a) and (b) follow. 

Now, if $\gamma_o+\gamma_+^*= 0$,  the sign of $(\Pi_o - \Pi_+^{-1}\circ \bar{\Pi}_o^{-1})(0)$ is determined by the sign of
\[
\left[
e^{\gamma_o\tau_o}\left(\dfrac{\pi t_+}{2(\beta_+^*)^3}\right)+\left(\dfrac{t_+}{d_+^*}-t_o\right)\left(1+e^{\gamma_o\pi}\right)\right](b_2+1).
\]
Thus, as $d_+^*>0$, $\beta_o>0$, $\beta_+^*>0$, $t_o=2\gamma_o \beta_o$ and $t_+=2\gamma_+^* \beta_+^*$ with $t_o t_+<0$, the statement ({c}) and (d) follow.
\end{proof}

The proof of Theorem \ref{the_paper2:01} follows directly from the proof of the next proposition.

\begin{proposition} 
\label{teo two limit cycle}
	Consider system {\rm (\ref{base system nf})}. Assume
	that  $X_-$ has a  center and $X_+$ and $X_o$ have  foci with different stability. Then  for $b_2<-1$, with $|b_2+1|$ small enough, if either $\gamma_o>0$ (equivalently $\gamma_+^*<0$) and $\gamma_o+\gamma_+^*>0$ or $\gamma_o<0$ (equivalently $\gamma_+^*>0$) and $\gamma_o+\gamma_+^*<0$,  system {\rm (\ref{base system nf})} has at least two limit cycles. Moreover, one limit cycle visit only the two zones $R_o$ and $R_+$ and the other visit the three zones. 
\end{proposition}
\begin{proof}
Firstly we consider the case $\gamma_o>0$, $\gamma_+^*<0$ and $\gamma_o+\gamma_+^*>0$. So, by Proposition \ref{prop.9.2},  $a_o^*-a_o^+>0$, i.e. we have the case showed in Figure \ref{figura8ab} (a). Note that the orbit $\Gamma(t)$, of system {\rm (\ref{base system nf})} by the point associated to $a_o^+$, spirals toward  the focus of $X_+$ when $t\rightarrow -\infty$. On the other hand, the focus of $X_+$ is an attractor. Then there is at least a limit cicle in two zones that pass by a point of $L_+^I$ between $a_+^*$ and $a_o^+$. Moreover this limit cycle is repeller. 
	
	In the three zones of Figure \ref{figura8ab} (a), we consider the Poincar\'e map $\Pi:[0, \infty)\rightarrow [0, \infty)$ given by  $\Pi=\bar\pi_o\circ\pi_+\circ\pi_o\circ\pi_-$. Now, by Proposition \ref{prop mau 1}~(b), $\pi_-$ is the identity map, so we
	can write $\Pi=\bar{\pi}_o\circ\pi_+\circ \pi_o$. Hence we have
	$$
	\Pi'(c)=
	\bar{\pi}_o'(\pi_+(\pi_o({c})))\cdot\pi_+'(\pi_o({c}))\cdot\pi_o'({c}).
	$$
From Propositions \ref{prop.8}, \ref{prop.9}, \ref{prop.10} and \ref{prop.11}, we obtain
\begin{equation}
\label{eq:15}
\displaystyle	\lim_{c\rightarrow +\infty}\Pi'(c)=\dfrac{b_2+1}{b_2-1}e^{\gamma_+\pi}\dfrac{b_2-1}{b_2+1}=e^{\gamma_+\pi}.
\end{equation}
As $\gamma_+^*<0$, so $\gamma_+<0$ for $|b_2+1|$ small enough and $\Pi$ is decreasing in a neighborhood of infinity, i.e. the infinity is a repeller to  system {\rm (\ref{base system nf})}. On the other hand, the orbit $\Gamma(t)$ spirals moving away from the focus. Therefore we have at least a limit cycle in the three zones.

For the second case $\gamma_o<0$, $\gamma_+^*>0$ and $\gamma_o+\gamma_+^*<0$, by Proposition \ref{prop.9.2}, we have  $a_o^*-a_o^+<0$, i.e. we have the case showed in Figure \ref{figura9ab} (b).  Hence, as the focus of $X_+$ is repeller, similar the previous case it follows that there is at least a limit cicle in two zones that pass by a point of $L_+^O$ between $a_+^*$ and $b_o^*$, which is attractor. Now, for Figure \ref{figura9ab} (b), we consider the same Poincar\'e map $\Pi$ as in the case of Figure \ref{figura8ab} (a).  However in this case the domain of $\Pi$ is $[c^{*}, \infty)$, where $\Pi(c^{*})=0$. Thus, as $\gamma_+^*>0$, it follows by \eqref{eq:15} that the infinity is an attractor  to  system {\rm (\ref{base system nf})} and as in the previous case we have at least a limit cycle in the three zones.

Then in both cases we conclude the existence  of at least two limit cycles, one visiting only two zones and the other visiting the three zones.
\end{proof}	

\begin{remark}
From above result, we conclude the existence of at least two limit cycles to system {\rm (\ref{base system nf})} assuming that  $X_-$ has a  center, $X_+$ and $X_o$ have  foci with different stability and $b_2<-1$, with $|b_2+1|$ small enough. The same conclusion we have for the equivalent case, when $X_+$ has a  center, $X_-$ and $X_o$ have  foci with different stability and $b_2>1$, with $|b_2-1|$ small enough. However we have not been able to determine the exact number of limit cycles for this class of vector fields. In fact, for the cases determined in  Figures \ref{figura8ab} (b) and \ref{figura9ab} (a) we also cannot say nothing about the existence or not of limit cycles for now.
\end{remark}

\section{Examples}
\label{sec examples}
In this section,  we will illustrate, through two examples, the appearance
of two limit cycles when the periodic annulus of Figures \ref{figura1a} is broken, obtaining the cases described in Figures \ref{figura8ab}(a) and \ref{figura9ab}(b). We assume that $X_+$ has a focus and, by the hypothesis (H1) and (H2),  $X_-$ has a center and $X_o$ has a focus  with different stability with
respect to the focus of $X_o$.  In this case, by the previous section, the phase portrait of  vector field \eqref{base system nf} is equivalent to Figures \ref{figura1a} if and only if $b_2=-1$ and $\gamma_++\gamma_o=0$.

The idea is to obtain two one-parameter families of vector fields, whose parameter is $b_2$, such that for $b_2=-1$ we have the periodic annulus of Figures \ref{figura1a}. On another words, we want to put the condition $\gamma_++\gamma_o=0$ as a function of $b_2$.  

By Lemma \ref{nf}, we have
\[
\gamma_o=\dfrac{a_1}{\sqrt{4-a_1^2}}\;\; \mbox{and}\;\;
 \gamma_+=\dfrac{c_{11}+a_1}{\sqrt{4(1+b_2-f_2+a_1c_{11})-(c_{11}+a_1)^2}}.
 \]
Now, for $b_2=-1$, $\gamma_o+\gamma_+=0$ if and only if
 $f_2=a_1c_{11}-\left(\dfrac{c_{11}+a_1}{a_1}\right)^2$. Hence,  doing
 $f_2=a_1c_{11}-\left(\dfrac{c_{11}+a_1}{a_1}\right)^2-\varepsilon$, a simple calculation gives us
\begin{equation}
\label{eq:16}
\gamma_+=\dfrac{c_{11}+a_1}{|c_{11}+a_1|}
\dfrac{|a_1|}{\sqrt{4-a_1^2+\dfrac{4a_1^2}{(c_{11}+a_1)^2}(b_2+1+\varepsilon)}}.
\end{equation}
Then, as $\gamma_o\gamma_+<0$ or equivalently $a_1(c_{11}+a_1)<0$, we get
\begin{equation}\label{eq:17}
\gamma_+=-
\dfrac{a_1}{\sqrt{4-a_1^2+\dfrac{4a_1^2}{(c_{11}+a_1)^2}(b_2+1+\varepsilon)}}.
\end{equation} 

{For}  $b_2+1+\varepsilon>0$, from \eqref{eq:17}, $\gamma_o+\gamma_+>0$ if $\gamma_o>0$ (i.e. $a_1>0$)  and $\gamma_o+\gamma_+<0$ if $\gamma_o<0$ (i.e. $a_1<0$).  Hence, from Proposition \ref{prop.9.2}  (for $b_2<-1$ with $|b_2+1|$ small enough), we have in the first case the configuration of Figure \ref{figura8ab}(a) and in the second one the configuration of Figure \ref{figura9ab} (b).  Moreover, by Proposition \ref{teo two limit cycle}, vector field \eqref{base system nf} has at least two limit cycles. 

We obtain the follows examples.

\begin{example}
\label{Ex:01}
Consider system \eqref{base system nf} with $\gamma_o>0$, $\gamma_+<0$ (i.e. $a_1>0$ and $c_{11}+a_1<0$), $4-a_1^2>0$, $a_{11}=-a_1$, $1-a_1^2-b_2+d_2>0$ and $f_2=a_1c_{11}-\left(\dfrac{c_{11}+a_1}{a_1}\right)^2-\varepsilon$, then (for $(b_2+1)^2+\varepsilon^2$ small enough) we have that
\begin{itemize}
\item[(a)] if $b_2>-1$ and $b_2+1+\varepsilon<0$, by Proposition 4.2 of \cite {Mauricio} and Proposition \ref{prop.9.2}, the phase portrait of vector field \eqref{base system nf} is equivalent to  Figure \ref{bifurca1}(a);
\item[(b)] if $b_2=-1$ and $\varepsilon=0$, by Proposition 12 from \cite{ClauWeMau}, the phase portrait of vector field \eqref{base system nf} is equivalent to  Figure \ref{bifurca1}(b);
\item[(c)] if $b_2<-1$ and $b_2+1+\varepsilon>0$, by Proposition \ref{teo two limit cycle}, the phase portrait of vector field \eqref{base system nf} in a neighborhood of these limit cycles is equivalent to  Figure \ref{bifurca1}({c}).
\end{itemize}
\begin{figure}[h!]
\begin{center}
\begin{overpic}[width=14cm,height=3cm]{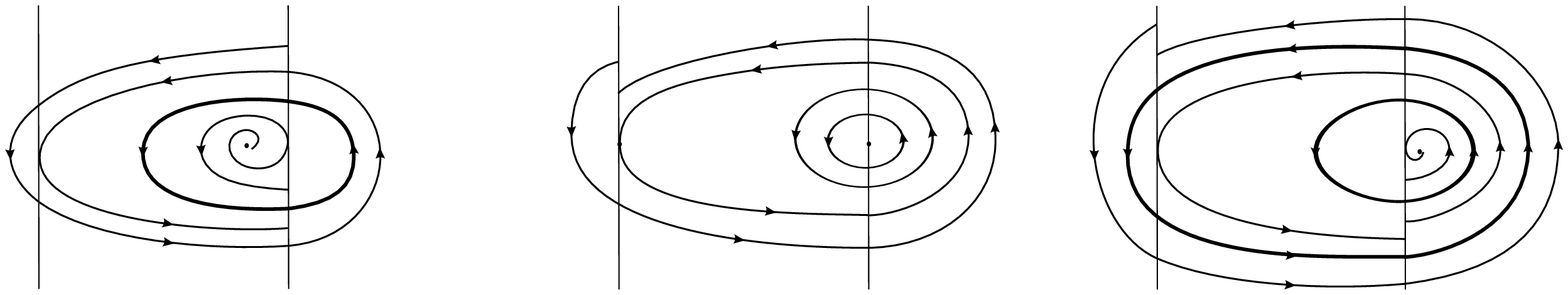}
\put(1,23){\tiny $L_-$} 
\put(17,23){\tiny $L_+$}
\put(38,23){\tiny $L_-$} 
\put(54,23){\tiny $L_+$}
\put(72.5,23){\tiny $L_-$} 
\put(89,23){\tiny $L_+$}
\put(-3,-6){\footnotesize (a) $b_2>-1$  and $b_2+1+\varepsilon<0$}
\put(39,-6){\footnotesize(b) $b_2=-1$ and $\varepsilon=0$}
\put(68,-6){\footnotesize({c}) $b_2<-1$ and $b_2+1+\varepsilon>0$}
\end{overpic}
\end{center}
\vspace{1.5cm}
\caption{Phase portrait of system \eqref{base system nf} with $a_1>0$ and $c_{11}+a_1<0$, $4-a_1^2>0$, $a_{11}=-a_1$, $1-a_1^2-b_2+d_2>0$ and $f_2=a_1c_{11}-\left(\dfrac{c_{11}+a_1}{a_1}\right)^2-\varepsilon$, for $(b_2+1)^2+\varepsilon^2$ small enough.} \label{bifurca1}
\end{figure}

We can use the software P5 to do numerical simulations of the phase portraits of system \eqref{base system nf} for specific parameter values. For instance, 
\begin{itemize}
\item the Figure \ref{Ex5.1Fig7}(a) correspond to Figure \ref{bifurca1}(a) with $a_1=1$,  $b_2=-0.9$, $c_{11}=-1.4$, 
$a_{11}=-1$, $d_2=4$ and $\varepsilon=-0.21$;
\item the Figure \ref{Ex5.1Fig7}(b) correspond to Figure \ref{bifurca1}(b) with $a_1=1$,  $b_2=-1$, $c_{11}=-1.4$, 
$a_{11}=-1$, $d_2=4$ and $\varepsilon=0$;
\item the Figure \ref{Ex5.1Fig7}({c}) correspond to Figure \ref{bifurca1}({c}) with $a_1=1$,  $b_2=-1.09$, $c_{11}=-1.4$, 
$a_{11}=-1$, $d_2=4$ and $\varepsilon=0.21$. 
\end{itemize}

\begin{figure}[h!]
\begin{center}
\begin{overpic}[width=15cm,height=5cm]{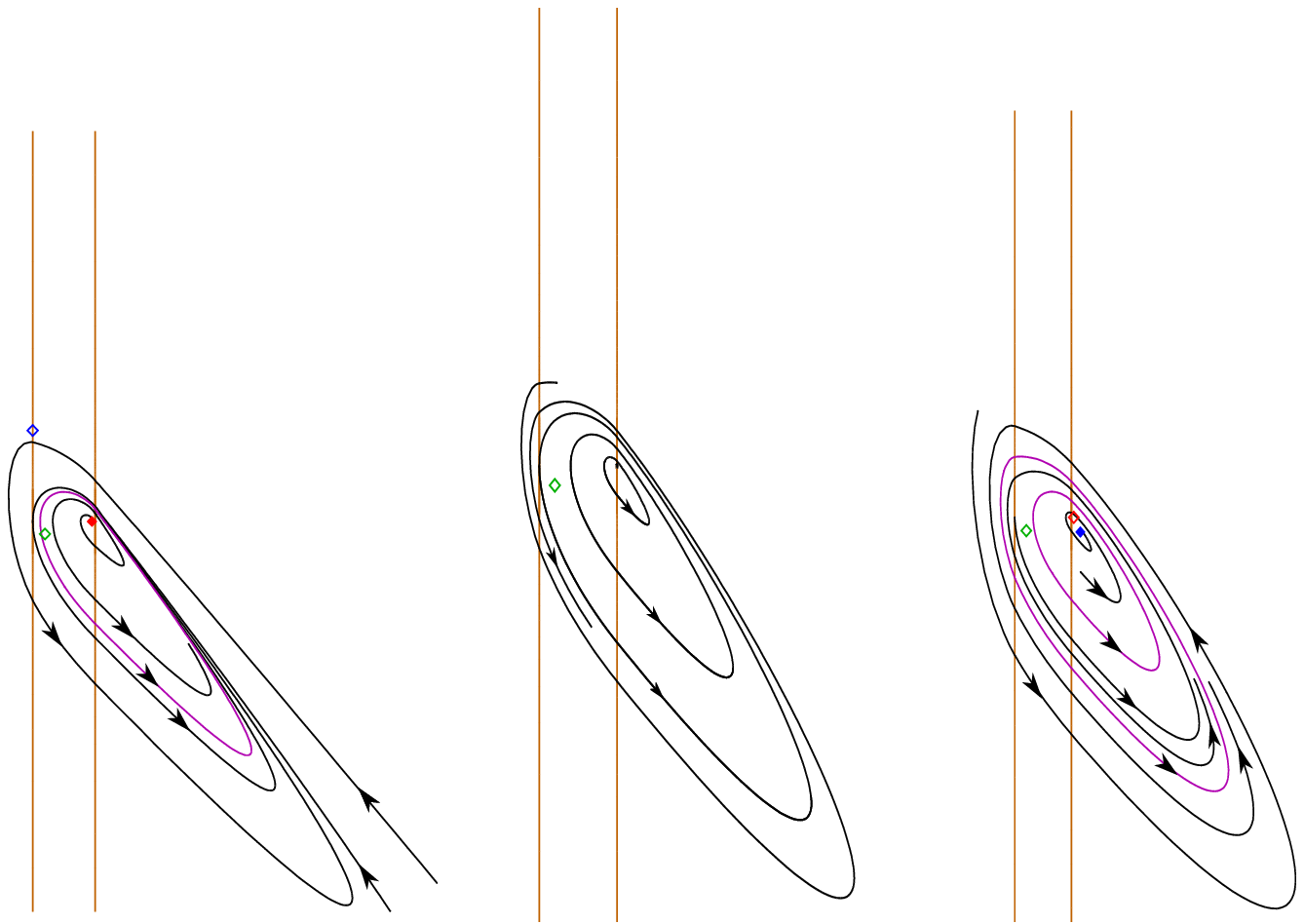}
\put(20.5,-6){(a)}
\put(50.7,-6){(b)}
\put(78,-6){({c})}
\end{overpic}
\end{center}
\vspace{1.5cm}
\caption{Numerical simulations of phase portraits of system \eqref{base system nf} using software P5 on the conditions of Example \ref{Ex:01}.} \label{Ex5.1Fig7}
\end{figure}
\end{example}

\begin{example}
\label{Ex:02}
Consider system \eqref{base system nf} with $\gamma_o<0$, $\gamma_+>0$ (i.e. $a_1<0$ and $c_{11}+a_1>0$), $4-a_1^2>0$, $a_{11}=-a_1$, $1-a_1^2-b_2+d_2>0$ and $f_2=a_1c_{11}-\left(\dfrac{c_{11}+a_1}{a_1}\right)^2-\varepsilon$, then (for $(b_2+1)^2+\varepsilon^2$) we have that
\begin{itemize}
\item[(a)] if $b_2>-1$  and $b_2+1+\varepsilon<0$, by Proposition 4.2 of \cite {Mauricio} and Proposition \ref{prop.9.2}, the phase portrait of vector field \eqref{base system nf} is equivalent to  Figure \ref{bifurca2}(a);
\item[(b)] if $b_2=-1$  and $\varepsilon=0$, by Proposition 12 from \cite{ClauWeMau}, the phase portrait of vector field \eqref{base system nf} is equivalent to  Figure \ref{bifurca2}(b);
\item[(c)] if $b_2<-1$  and $b_2+1+\varepsilon>0$, by Proposition \ref{teo two limit cycle}, the phase portrait of vector field \eqref{base system nf} in a neighborhood of these limit cycles is equivalent to  Figure \ref{bifurca2}({c}).
\end{itemize}

\begin{figure}[h!]
\begin{center}
\begin{overpic}[width=14cm,height=3cm]{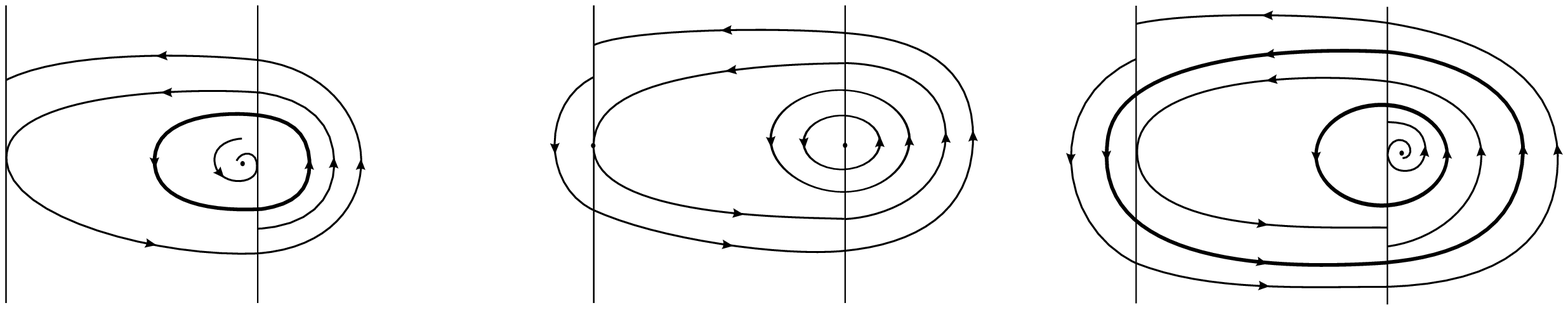}
\put(-1,23){\tiny $L_-$} 
\put(15,23){\tiny $L_+$}
\put(36.5,23){\tiny $L_-$} 
\put(52.5,23){\tiny $L_+$}
\put(71.5,23){\tiny $L_-$} 
\put(88,23){\tiny $L_+$}
\put(-5,-6){\footnotesize (a) $b_2>-1$  and $b_2+1+\varepsilon<0$}
\put(36,-6){\footnotesize(b) $b_2=-1$ and $\varepsilon=0$}
\put(68,-6){\footnotesize({c}) $b_2<-1$ and $b_2+1+\varepsilon>0$}
\end{overpic}
\end{center}
\vspace{1.5cm}
\caption{Phase portrait of system \eqref{base system nf} with $a_1<0$ and $c_{11}+a_1>0$, $4-a_1^2>0$, $a_{11}=-a_1$, $1-a_1^2-b_2+d_2>0$ and $f_2=a_1c_{11}-\left(\dfrac{c_{11}+a_1}{a_1}\right)^2+2(b_2+1)$, for $(b_2+1)^2+\varepsilon^2$ small enough.} \label{bifurca2}
\end{figure}

As in Example \ref{Ex:01}, we can use the software P5 (see \cite{P5}) to do numerical simulations of the phase portraits of system \eqref{base system nf} for specific parameter values. For instance, 
\begin{itemize}
\item the Figure \ref{Ex5.2Fig8}(a) correspond to Figure \ref{bifurca2}(a) with $a_1=-0.8$,  $b_2=-0.9$, $c_{11}=1.4$, 
$a_{11}=0.8$, $d_2=4$ and $\varepsilon=-0.35$;
\item the Figure \ref{Ex5.2Fig8}(b) correspond to Figure \ref{bifurca2}(b) with $a_1=-0.8$,  $b_2=-1$, $c_{11}=1.4$, 
$a_{11}=0.8$, $d_2=4$ and $\varepsilon=0$;
\item the Figure \ref{Ex5.2Fig8}({c}) correspond to Figure \ref{bifurca2}({c}) with $a_1=-0.8$,  $b_2=-1.09$, $c_{11}=1.4$, 
$a_{11}=0.8$, $d_2=4$ and $\varepsilon=0.43$. 
\end{itemize}

\begin{figure}[h!]
\begin{center}
\begin{overpic}[width=15cm,height=5cm]{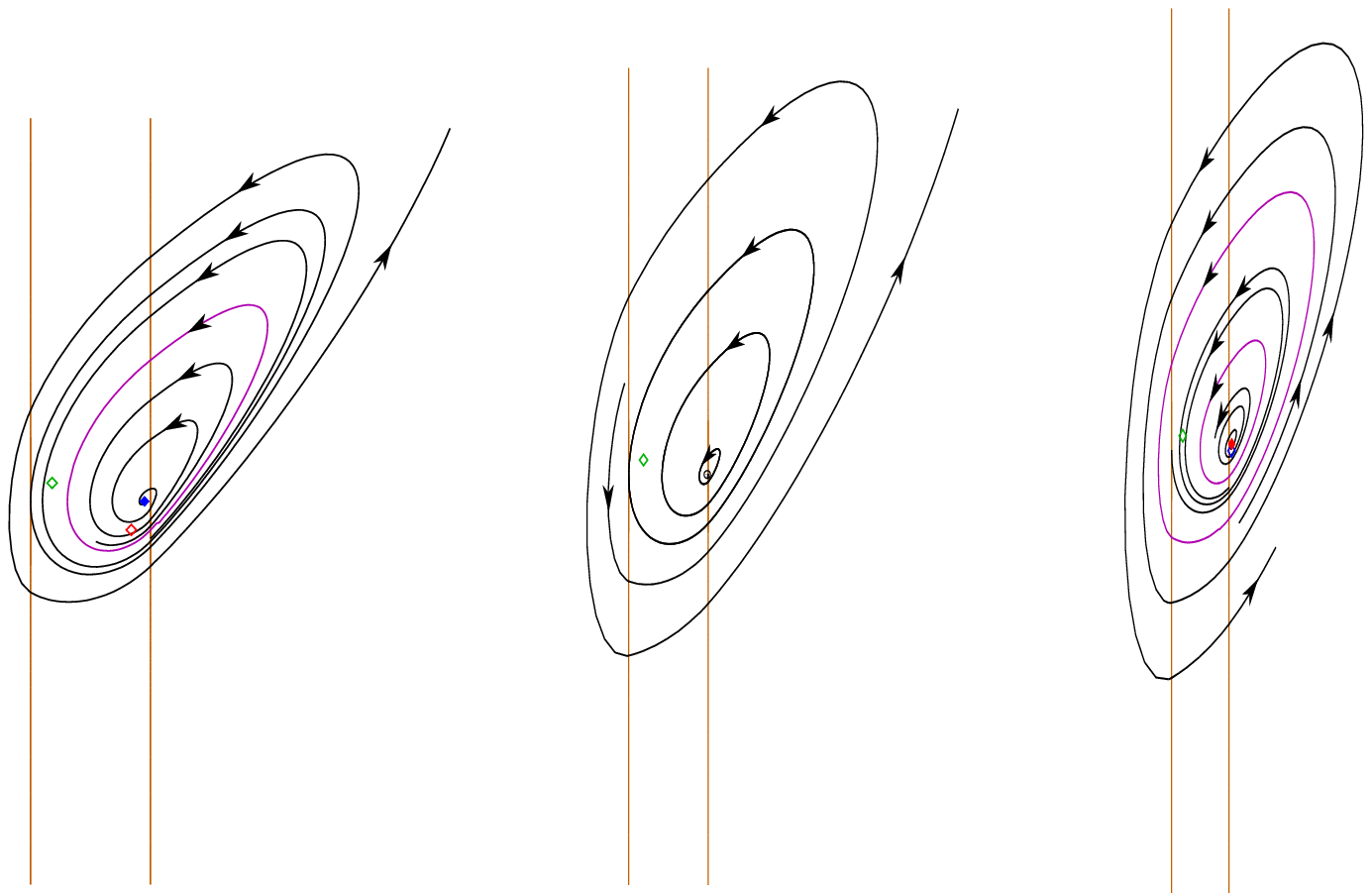}
\put(20.5,-6){(a)}
\put(50.7,-6){(b)}
\put(78,-6){({c})}
\end{overpic}
\end{center}
\vspace{1.5cm}
\caption{Numerical simulations of phase portraits of system \eqref{base system nf} using software P5 on the conditions of Example \ref{Ex:02}.} \label{Ex5.2Fig8}
\end{figure}
\end{example}

\section*{Acknowledgments}

The first author is partially supported by Fapesp grant number
2013/15941-5 and grant FP7-PEOPLE-2012-IRSES 318999. The second and third authors are partially supported by a FAPESP grant 2013/34541--0. The second author is supported by a CAPES PROCAD grant 88881.068462/2014-01.

\end{document}